\documentclass[a4paper,10pt]{article}
\usepackage{QED}
\usepackage{amsmath}
\usepackage{amssymb}
\usepackage{a4}

\def\b{\beta}
\def\g{\gamma}

\def\G{\Gamma}

\def\t{\tau}

\def\d{\delta}

\def\l{\lambda}

\def\s{\sigma}

\def\f{\rightarrow}

\def\v{\vdash}

\def\et{\wedge}
\def\<{\langle}
\def\>{\rangle}
\def\F{\displaystyle\frac}
\newtheorem{theorem}{Theorem}[section]
\newtheorem{lemma}{Lemma}[section]

\newtheorem{corollary}{Corollary}[section]

\newtheorem{definition}{Definition}[section]
\newtheorem{notation}{Notation}[section]

\begin{document}

\title{A short proof that adding some permutation rules to $\b$ preserves $SN$}

\author{Ren\'e David \\
LAMA - Equipe LIMD -
 Universit\'e de Chamb\'ery\\
e-mail : rene.david@univ-savoie.fr }

%\date{}

\maketitle

\begin{abstract}
I show that, if a term is $SN$ for $\beta$,  it remains $SN$ when
some permutation rules  are added.
\end{abstract}

\section{Introduction}

Strong normalization  (abbreviated as $SN$) is a property of
rewriting systems that is often desired. Since about 10 years many
researchers have considered the following question : If a
$\l$-term is $SN$ for the $\b$-reduction, does it remain $SN$ if
some other reduction rules are added ?  They are mainly interested
with permutation rules they introduce to be able to delay some
$\b$-reductions in, for example, {\it let $x$ = ... in ...}
constructions or in {\it calculi with explicit substitutions}.
Here are some papers considering such permutations rules: L.
Regnier \cite{regnier}, F Kamareddine \cite{fairouz}, E. Moggi
\cite{moggi}, R. Dyckhoff and S. Lengrand \cite{lengrand}, A. J.
Kfoury  and J. B. Wells \cite{wells}, Y. Ohta and M. Hasegawa
\cite{ohta}, J. Espírito Santo \cite{jose-1} and \cite{jose-2}.

Most of these papers show that $SN$ is preserved by the addition
of the permutation rules they introduce. But these proofs are
quite long and complicated or need some restrictions to the rule.
For example the rule  $( M \ (\lambda x. N \ P ) ) \triangleright
  (\lambda x. (M \ N) \ P )$ is often restricted to the case when
  $M$ is an abstraction (in this case it is usually called
  ${assoc}$).

  I give here a very simple proof that the permutations rules
  preserve $SN$ when they are added all together and with no
  restriction. It is done as follows. I show
  that every term which is typable in the system (often called system $\cal{D}$) of
  types built with $\rightarrow$ and $\et$ is strongly normalizing for all the rules
($\beta$ and the permutation rules). Since it is well known that a
term is $SN$ for the $\beta$-rule iff it is typable in this
system, the result follows.

\section{Definitions and notations}

\begin{definition}
\begin{itemize}
  \item The set of $\l$-terms is defined by the following grammar

$${\cal M} := x\ | \ \l x.  {\cal M} \ | \ ({\cal M} \; {\cal M})$$

  \item  The set ${\cal T}$ of types is defined by the following
  grammar where  ${\cal A}$ is a set of atomic constants

$${\cal T} ::= \; {\cal A} \; \mid
{\cal T} \f {\cal T} \;  \mid {\cal T} \et {\cal T}$$

\item The typing rules are the following :

\begin{center}

$\F{}{\G , x : A \v x : A} $

\vspace{.5cm}

$\F{\G \v M : A \f B \quad \G \v N : A} {\G \v (M \; N) : B }$
\hspace{0.5cm} $\F{\G, x: A \v M : B} {\G \v \l x.M : A \f B}$

\vspace{.5cm}

$\F{\G \v M : A \et B} {\G \v M : A }$ \hspace{0.5cm} $\F{\G \v M
: A \et B} {\G \v M : B }$

\vspace{.5cm}

$\F{\G \v M : A  \quad \G \v M : B} {\G \v M : A \et B }$

\end{center}
\end{itemize}

\end{definition}

\begin{definition}
The reduction rules are the following.
\begin{itemize}
  \item $\beta$ : $(\lambda x. M \ N) \triangleright M[x:=N]$
  \item $\d$ : $(\l y.\l x. M \ N) \triangleright \l x.(\l y. M \
  N)$
\item $\g$ : $(\lambda x. M \ N \ P) \triangleright (\lambda x. (M \ P) \
N)$
   \item ${assoc}$  : $( M \ (\lambda x. N \ P ) ) \triangleright
  (\lambda x. (M \ N) \ P )$
\end{itemize}
\end{definition}

 Using
Barendregt's convention for the names of variables, we  assume
that, in $\g$ (resp. $\d$, $assoc$), $x$ is not free in $P$ (resp.
in $N$,  in $M$).

The rules $\d$ and $\g$ have been introduced by Regnier in
\cite{regnier} and are called there the $\s$-reduction. It seems
that the first formulation of {\em assoc} appears in Moggi
\cite{moggi} in the restricted case where  $M$ is an abstraction
and in a  ``{\em let ... in ...}'' formulation.

\begin{notation}
\begin{itemize}

  \item If $t$ is a term, $size(t)$ denotes its size and $type(t)$ the size of its type.
  If $t \in SN$ (i.e. every sequence of reductions starting from $t$ is finite), $\eta(t)$  denotes
   the length of the
longest reduction of $t$.
\item Let $\s$ be a substitution. We say that $\s$ is fair if the  $\s(x)$ for $x
\in dom(\s)$ all have the same type (that will be denoted as
$type(\s)$). We say that $\s \in SN$  if, for each $x \in
dom(\s)$, $\s(x) \in SN$.
\item Let $\s \in SN$ be a substitution and $t$ be a term. We denote by $size(\s,t)$ (resp. $\eta(\s,t)$) the
sum, over $x \in dom(\s)$, of $nb(t,x). size(\s(x))$ (resp.
$nb(t,x). \eta(\s(x))$)
 where $nb(t,x)$ is the number of occurrences of $x$
in $t$.
\item If $\overrightarrow{M}$ is a sequence of terms, $lg(\overrightarrow{M})$ denotes its length,
$M(i)$  denotes the $i$-th element of the sequence and
$tail(\overrightarrow{M})$  denotes $\overrightarrow{M}$ from
which the first element has been deleted.
 \item Assume $t=(H \  \overrightarrow{M})$ where $H$ is an abstraction or a variable and $lg(\overrightarrow{M})\geq
 1$.

\begin{itemize}

\item If $H$ is an abstraction (in this case we say that $t$ is $\b$-head reducible), then $M(1)$ will
be denoted as $Arg[t]$ and
 $(R' \ tail(\overrightarrow{M}))$ will
be denoted by $B[t]$ where $R'$ is the reduct of the $\beta$-redex
$(H  \ Arg[t])$.
\item  If $H=\l x. N$ and $lg( \overrightarrow{M}) \geq 2$ (in this case we say that
$t$ is $\g$-head reducible), then $(\l x. (N \ M(2))$  $ M(1) \
M(3) \ ... \ M(lg(\overrightarrow{M})))$ will be denoted by
$C[t]$.
\item If $H=\l x. \l y. N$  (in this case we say that
$t$ is $\d$-head reducible), then $(\l y. (\l x. N \ M(1)) \ M(2)
\ ... \ M(lg(\overrightarrow{M})))$ will be denoted by $D[t]$.
\item If $M(i)=(\lambda x. N \ P)$, then the term
$(\lambda x. (H \ M(1) \ ... \ M(i-1) \ N ) \ P \ M(i+1) \ ... \
M(lg(\overrightarrow{M})))$ will be denoted by $A[t, i]$ and we
say that $M(i)$ is the redex put in head position.

\end{itemize}
\item Finally, in a proof by induction, IH will denote the induction
hypothesis.

\end{itemize}

\end{notation}

\section{The theorem}

\begin{theorem}
Let $t$ be a term. Assume $t$ is strongly normalizing for $\b$.
Then $t$ is strongly normalizing for $\b$, $\d$, $\g$ and $assoc$.
\end{theorem}
\begin{proof}
This follows immediately from Theorem \ref{D} and corollary
\ref{cor} below.
\end{proof}

\begin{theorem}\label{D}
A term is $SN$ for the $\b$-rule iff it is typable in  system
$\cal{D}$.
\end{theorem}
\begin{proof}
This is a classical result. For the sake of completeness I recall
here the proof of the only if direction given in  \cite{moi}. Note
that it is the only direction that is used in this paper and that
corollary \ref{cor} below actually gives the other direction. The
proof is by induction on $\langle\eta(t),size(t)\rangle$.

-  If $t=\lambda x\;u.$ This follows immediately from the IH.

- If $t=(x\;v_{1}\;...\;v_{n})$.
 By the IH, for every $j$, let $x:A_{j},\Gamma _{j}\vdash v_{j}:B_{j}$. Then $x:\bigwedge
 A_{j}\et
(B_{1},...,B_{n}\rightarrow C),$ $\bigwedge \Gamma _{j}\vdash t:C$
where $C$ is any type, for example any atomic type.

-  If $t=(\lambda x.a \;b\;\overrightarrow{c})$. By the IH,
$(a[x:=b]\;\overrightarrow{c})$ is typable.  If $x$ occurs in $a$,
let $A_{1}\;...\;A_{n}$ be the types of the occurrences of $b$ in
the typing of $(a[x:=b]\;\overrightarrow{c})$. Then $ t $ is
typable by giving to $x$ and $b$ the type $A_{1}\;\et ...\;\et
A_{n}$. Otherwise, by the induction hypothesis $b$ is typable of
type $B$ and then $t$ is typable by giving to $x$ the type $B$.
\end{proof}

{\it From now on, $\triangleright$ denotes the reduction by one of
the rules $\b$, $\d$, $\g$ and $assoc$}.

\begin{lemma}\label{prepa}
\begin{enumerate}
\item The system satisfies subject reduction i.e. if $\G \v t : A$ and $t \triangleright t'$ then
$\G \v t' : A$.

\item If $t \triangleright t'$ then $t[x:=u] \triangleright t'[x:=u]$.
  \item If   $t'=t[x:=u] \in SN$ then  $t \in SN$
  and $\eta(t) \leq \eta(t')$.
\end{enumerate}
\end{lemma}
\begin{proof}
Immediate.
\end{proof}

\begin{lemma}\label{cs-sn}
Let $t=(H \  \overrightarrow{M})$ be such that $H$ is an
abstraction or a variable and $lg(\overrightarrow{M})\geq 1$.
Assume that
\begin{enumerate}
\item If $t$ is $\d$-head reducible (resp. $\g$-head reducible, $\b$-head reducible), then $D[t] \in SN$
(resp. $C[t] \in SN$, $Arg[t], B[t] \in SN$).

\item For each $i$ such that $M(i)$is a redex, $A[t,i] \in SN$,

\end{enumerate}
Then $t \in SN$.
\end{lemma}
\begin{proof}
By induction on $\eta(H) + \sum \eta(M(i))$. Show that each reduct
of $t$ is in $SN$.
\end{proof}

\begin{lemma}\label{prepa2}
If $(t \ \overrightarrow{u}) \in SN$ then $(\lambda x. t \ x \
\overrightarrow{u}) \in  SN$.
\end{lemma}
\begin{proof}
This is a special case of the following result. If $t \in SN$ then
so is $F(t)$ where  $F(t)$ is obtained in  the following way:
choose a node on the left branch of $t$ and replace the sub-term
$u$ at this node by $(\l x. u \ x)$. The proof is by induction on
$\langle type(u), \eta(t), size(t)\rangle$, using Lemma
\ref{cs-sn}. The only non immediate cases are when the head redex
has been created by the transformation $F$. The case of $\b$ is
trivial. For $\d$ and $\g$, the result follows from the fact that
the type of the sub-term modified has decreased and there is
nothing to prove for {\em assoc} since the the change is in the
left branch.
\end{proof}

\begin{theorem}\label{main}
Let $t\in SN$ and $\s \in SN$ be a fair substitution. Then $\s(t)
\in SN$.
\end{theorem}
\begin{proof}
By induction on $\langle type(\s),\eta(t),size(t), \eta(\s,t),
size(\s,t) \rangle$. If $t$ is an abstraction or a variable the
result is trivial. Thus assume $t=(H \ \overrightarrow{M})$ where
$H$ is an abstraction or a variable and
$n=lg(\overrightarrow{M})\geq 1$. Let
$\overrightarrow{N}=\s(\overrightarrow{M})$.

\noindent {\em Claim} : Let $\overrightarrow{P}$ be a (strict)
initial or
a final sub-sequence of $\overrightarrow{N}$. Then $(z \ \overrightarrow{P}) \in SN$. \\
{\em Proof} : This follows immediately from Lemma \ref{prepa} and
the IH. \hspace{2.8cm}$\Box$

\medskip

We use Lemma \ref{cs-sn} to show that $\s(t) \in SN$.
\begin{enumerate}
\item Assume $\s(t)$ is $\d$-head reducible. We have to show that
$D[\s(t)] \in SN$. There are 3 cases to consider.
\begin{enumerate}
  \item If $t$ was already $\d$-head reducible, then
  $D[\s(t)]=\s(D[t])$ and the result follows from the IH.
  \item If $H$ is a variable and $\s(H)=\l x. \l y.a$, then
  $D[\s(t)]=t'[z:=\l y. (\l x.a \ N(1))]$ where $t'=(z\ tail(\overrightarrow{N}))$.
  By the claim, $t'\in SN$ and since $type(z) <
  type(\s)$ it is enough to check that $\l y. (\l x.a \ N(1))
  \in SN$. But this is $\l y. (z' \ N(1))[z':=\l x. a]$. But, by the claim, $(z' \ N(1))\in SN$ and we conclude by the IH since
  $type(z') < type(\s)$.
  \item If $H=\l x. z$ and $\s(z)=\l y. a$, then $D[\s(t)]=(\l y. (\l x. a \ N(1)) \ tail(\overrightarrow{N}))
  =\t(t')$ where $t'=(z' \ tail(\overrightarrow{M}))$ and $\t$ is the same as $\s$ on the variables of
 $tail(\overrightarrow{M})$ and $\t(z')=\l y.(\l x. a \ N(1))$. By the IH, it is enough to show
  that $(\l x. a \ N(1))\in SN$. But this is $(\l x. z'' \
  N(1))[z'':=a]$ and, since $type(a) < type(\s)$ it is enough to
  show that $u=(\l x. z'' \ N(1))=\s'(t'') \in SN$ where $t''$ is a sub-term of $t$ (up to
  the renaming of $z$ into $z''$) and $\s'$ is as $\s$ but $z'' \not\in dom(\s')$.
This follows from the IH since $size(\s',t'')<size(\s,t) $.
\end{enumerate}

\item Assume $\s(t)$ is $\g$-head reducible. We have to show that
$C[\s(t)] \in SN$. There are 4 cases to consider.

\begin{enumerate}
  \item If $H$ is an abstraction, then $C[\s(t)]=\s(C[t])$ and the result follows immediately from
the IH.
\item  $H$ is a variable and $\s(H)=\l y. a$, then
$C[\s(t)]=(\l y.(a\ N(2))\ N(1) \ N(3)$ $... \ N(n))= (\l y.(a\
N(2))\ y \ N(3)$ $ ... \ N(n))[y:=N(1)]$. Since $type(N(1)) <
type(\s)$, it is enough, by the IH, to show $(\l y.(a\ N(2))\ y \
N(3)$ $... \ N(n)) \in SN$ and so, by Lemma \ref{prepa2}, that
$u=(a\ N(2) \ N(3) \ ... \ N(n)) \in SN$.  By the claim,  $(z \
tail(\overrightarrow{N})) \in SN$ and the result follows from the
IH since $u=(z \ tail(\overrightarrow{N}))[z:=a]$ and $type(a) <
type(\s)$.
\item $H$ is a variable and $\s(H)=(\l y. a \ b)$,  then
$C[\s(t)]=(\l y.(a\ N(1))\ b$ $N(2) \ ... \ N(n))=(z \
tail(\overrightarrow{N}))[z:= (\l y.(a\ N(1))\ b)]$. Since
$type(z) < type(\s)$, by the IH it is enough to show that $u=(\l
y.(a\ N(1))\ b) \in SN$. We use Lemma \ref{cs-sn}.

- We first have to show that $B[u]\in SN$. But this is $(a[y:=b] \
N(1))$ which is in $SN$ since $u_1=(a[y:=b] \ \overrightarrow{N})
\in SN$ since $u_1=\t(t_1)$ where  $t_1$ is the same as $t$ but
where we have given to the variable $H$  the fresh name $z$, $\t$
is the same as $\s$ for the variables in $dom(\s)$ and
$\t(z)=a[y:=b]$ and thus we may conclude by the IH since $\eta(\t,
t ) < \eta(\s, t)$.

- We then have to show that, if $b$ is a redex say $(\l z. b_1 \
b_2 )$, then $A[u,1]=(\l z. (\l y. a \ N(1) \ b_1) \ b_2) \in SN$.
Let $u_2=\t(t_2)$ where $t_2$ is the same as $t$ but where we have
given to the variable $H$ the fresh name $z$, $\t$ is the same as
$\s$ for the variables in $dom(\s)$ and $\t(z)=\s(A[H,1])$. By the
IH, $u_2 \in SN$. But $u_2=(\l z. (\l y.a \ b_1) \ b_2 \
\overrightarrow{N})$ and thus $u_3=(\l z. (\l y.a \ b_1) \ b_2 \ \
N(1)) \in SN$.  Since $u_3$ reduces to $A[u,1]$ by using twice by
the $\g$ rule, it follows that $A[u,1] \in SN$.
  \item If $H$ is a variable and $\s(H)$ is $\g$-head reducible, then
  $C[\s(t)]=\t(t')$ where $t'$ is the
same as $t$ but where we have given to the variable $H$  the fresh
name $z$ and $\t$ is the same as $\s$ for the variables in
$dom(\s)$ and $\t(z)=\s(C[H])$. The result follows then from the
IH.

\end{enumerate}

\item Assume that $\s(t)$ is $\b$-head reducible. We have to show that $Arg[\s(t)]\in SN$ and
 that $B[\s(t)] \in SN$. There are 3 cases to consider.

\begin{enumerate}
  \item If $H$ is an abstraction, the result follows immediately from
the IH since then $Arg[\s(t)]= \s(Arg[t])$ and
$B[\s(t)]=\s(B[t])$.

  \item If $H$ is a variable  and $\s(H)=\l y. v$
for some $v$. Then $Arg[\s(t)]= N(1) \in SN $ by the IH and
$B[\s(t)]=(v[y:=N(1)] \ tail(\overrightarrow{N}))=(z \
tail(\overrightarrow{N}))[z:=v[y:=N(1)]]$. By the claim, $(z \
tail(\overrightarrow{N})) \in SN$. By the IH, $v[y:=N(1)] \in SN$
since  $type(N(1)) < type(\s)$.  Finally the IH implies that
$B[\s(t)] \in SN$ since $type(v) < type(\s)$.

  \item $H$ is a variable  and $\s(H)=(R
\ \overrightarrow{M'})$ where $R$ is a $\b$-redex. Then
$Arg[\s(t)]=Arg[\s(H)] \in SN$ and $B[\s(t)]=(R' \
\overrightarrow{M'} \ \overrightarrow{N})$ where $R'$ is the
reduct of $R$. But then $B[\s(t)] =\t(t')$  and $t'$ is the same
as $t$ but where we have given to the variable $H$  the fresh name
$z$ and $\t$ is the same as $\s$ for the variables in $dom(\s)$
and $\t(z)=(R' \ \overrightarrow{M'})$. We conclude by the IH
since $\eta(\t,t') < \eta(\s,t)$.
\end{enumerate}

\item We, finally, have to show that, for each $i$, $A[\s(t), i] \in
SN$. There are again 3 cases to consider.

\begin{enumerate}

  \item If the redex put in head position is some $N(j)$ and
  $M(j)$ was already a redex. Then $A[\s(t), i]=\s(A[t,j])$ and
  the result follows from the IH.

  \item If the redex put in head position is some $N(j)$ and
  $M(j)=(x \ a)$ and $\s(x)=\l y. b$ then $A[\s(t), i]= \l y. (\s(H) \
N(1) \ ... \ N(j-1) \ b) \ \s(a) \ N(j+1) \ ... \ N(n))$. Since
$type(\s(a)) < type(\s)$ it is enough, by the IH, to show that $\l
y. (\s(H) \ N(1) \ ... \ N(j-1) \ b) \ y \ N(j+1) \ ... \ N(n))$
and so, by Lemma \ref{prepa2}, that $(\s(H) \ N(1) \ ... \ N(j-1)
\ b \ N(j+1) \ ... \ N(n)) \in SN$. Since $type(b) < type(\s)$ it
is enough to show $u=(\s(H) \ N(1) \ ... \ N(j-1) \ z \ N(j+1) \
... \ N(n)) \in SN$. Let $t'= (H \ \overrightarrow{M'})$ where
$\overrightarrow{M'}$ is defined by $M'(k)=M(k)$, for $k \neq j$,
$M'(j)=z$.  Since $t= t'[z:=(x \ a)]$ and $u=\s(t')$ the result
follows from Lemma \ref{prepa} and the IH.

\item If,  finally, $H$ is a variable, $\s(H)=(H' \
\overrightarrow{M'})$ and the redex put in head position is some
$M'(j)$. Then, $A[\s(t),i]=\t(A[t',j])$ where $t'$ is the same as
$t$ but where we have given to the variable $H$ the fresh variable
$z$ and $\t$ is the same as $\s$ for the variables in $dom(\s)$
and $\t(z)=A[\s(H),j]$. We conclude by the IH since $\eta(\t,t') <
\eta(\s,t)$.

  \end{enumerate}

\end{enumerate}
\end{proof}

\begin{corollary}\label{cor}
Let $t$ be a typable  term. Then $t$ is strongly normalizing.
\end{corollary}
\begin{proof}
By induction on $size(t).$ If $t$ is an abstraction or a variable
the result is trivial.  Otherwise $t=(u\ v)=(x\ y)[x:=u][y:=v]$
and the result follows immediately  from Theorem \ref{main} and
the IH.
\end{proof}

\end{document}